\documentclass{amsart}
\usepackage{amsmath,amssymb,amsfonts}
\usepackage[mathscr]{eucal}

\usepackage{enumerate}

\theoremstyle{plain}
\newtheorem{theorem}{Theorem}[section]
\newtheorem{cor}[theorem]{Corollary}
\newtheorem{prop}[theorem]{Proposition}
\newtheorem{lemma}[theorem]{Lemma}

\theoremstyle{definition}
\newtheorem{definition}[theorem]{Definition}

\newtheorem{remark}[theorem]{Remark}
\newtheorem*{thank}{Acknowledgments}

\input xy
\xyoption{all}

\newcommand{\Deltaop}{{\bf \Delta}^{op}}

\newcommand{\SSets}{\mathcal{SS}ets}

\newcommand{\Map}{\text{Map}}
\newcommand{\map}{\text{map}}
\newcommand{\Hom}{\text{Hom}}
\newcommand{\colim}{\text{colim}}

\newcommand{\Sets}{\mathcal{S}ets}

\newcommand{\Secat}{\mathcal Se \mathcal Cat}

\newcommand{\css}{\mathcal{CSS}}

\newcommand{\Ho}{\text{Ho}}

\newcommand{\qcat}{\mathcal{QC}at}
\newcommand{\orbit}{\mathcal O^{op}_G}

\begin{document}

\title{Equivalence of models for equivariant $(\infty, 1)$-categories}

\author[J.E.\ Bergner]{Julia E.\ Bergner}

\email{bergnerj@member.ams.org}

\address{Department of Mathematics, University of California, Riverside, CA 92521}

\date{\today}



\thanks{The author was partially supported by NSF grant DMS-1105766.  This research was done while the author was in residence at MSRI during Spring 2014, supported by NSF grant 0932078 000.}

\begin{abstract}
In this paper we show that the known models for $(\infty, 1)$-categories can all be extended to equivariant versions for any discrete group $G$.  We show that in two of the models we can also consider actions of any simplicial group $G$.
\end{abstract}

\maketitle

\section{Introduction}

Two areas of much recent research in homotopy theory have been the development and application of homotopical categories, or $(\infty, 1)$-categories, and equivariant homotopy theory.  In this paper, we seek to bring the two ideas together and investigate equivariant $(\infty, 1)$-categories.

There are many different models for $(\infty, 1)$-categories, all known to be equivalent in the sense that their respective model categories are all Quillen equivalent.  Simplicial or topological categories are categories with a simplicial set or space of morphisms between any two objects.  Segal categories and complete Segal spaces are bisimplicial sets having properties resembling the simplicial nerve of a simplicial category, but with a weaker form of composition.  Quasi-categories encode the same information in a simplicial set.  Each model for $(\infty, 1)$-categories can be thought of as living in its respective model category, in which a weak equivalence is defined by a homotopical version of equivalence of categories.

Given a group $G$, a $G$-equivariant $(\infty, 1)$-category should be one of the above structures equipped with an action of $G$.  The morphisms in the category of such should respect the $G$-action, and a weak equivalence should be a map which induces weak equivalences on $H$-fixed point objects for all (closed) subgroups $H$ of $G$.

To prove that each model for $(\infty, 1)$-categories has a good equivariant analogue, and that these models are in turn all equivalent to one another, we apply general tools of Stephan  \cite{stephan} and of Bohmann, Mazur, Osorno, Ozornova, Ponto, and Yarnall \cite{bmoopy}.  Their results give conditions under which a model category has a corresponding $G$-model category, and under which this process preserves Quillen equivalences.  A related approach has also been described by Dotto and Moi \cite{dm}.

For complete Segal spaces and quasi-categories, establishing that these conditions hold for an action by a discrete group $G$ is immediate from Stephan's results.  For simplicial categories and Segal categories, we give proofs here.  These conditions also guarantee that we can carry over all the Quillen equivalences that we have non-equivariantly.

When $G$ is a simplicial group, then we need the original model category to have the structure of a simplicial model category.  This additional assumption holds for the complete Segal space and Segal category models, so we establish equivariant versions in these contexts.

It is perhaps more interesting to consider the case where $G$ is a topological group, and in particular a compact Lie group.  In this case, we need to begin with a topological (rather than simplicial) model category, so it is expected that the right models are complete Segal spaces and Segal categories taken in topological spaces rather than in simplicial sets.  We pursue this perspective, as well as a more hands-on approach to the equivariant structures, in forthcoming work with Chadwick.  

We expect that many examples of $(\infty, 1)$-categories have natural group actions that can be investigated from this perspective.  For example, any monoidal $(\infty, 1)$-category has an action of the automorphism group of a unit object.

The methods that we use in this paper can be extended to more general settings, such as equivariant $(\infty, n)$-categories and equivariant $(\infty,1)$-operads.  Again, for many of the known models we get equivariant versions immediately from Stephan's results; for the generalizations of simplicial categories (categories enriched in $\Theta_n$-spaces and simplicial operads, respectively), the proofs can be obtained by modifying the one given here appropriately.  The framework of equivariant $(\infty, n)$-categories should give a context in which to understand equivariant (extended) topological field theories.  A further extension is to the 2-Segal spaces of Dyckerhoff and Kapranov \cite{dk}; as these structures often arise from $S_\bullet$-constructions there are a number of examples to be investigated with group actions incorporated.  We plan to look at such examples in future work.

We begin in Section 2 with a statement of the general theorems that we wish to use, in the case where $G$ is a discrete group.  In Section 3, we show that they can be applied immediately to the complete Segal space and quasi-category models. In Sections 4 and 5, we prove that they can be applied to the simplicial category and Segal category models.  Lastly, in Section 6 we consider simplicial groups $G$ and show that we get equivariant versions of the complete Segal space and Segal category models in this case.

\begin{thank}
The author would like to thank Anna Marie Bohmann, Ang\'elica Osorno, and Marc Stephan for several helpful conversations about this work.
\end{thank}

\section{Equivariant versions of model categories}

In this section, we summarize results of Stephan, namely, sufficient conditions under which a $G$-equivariant version of a cofibrantly generated model category $\mathcal C$ exists \cite{stephan}.  To begin, we assume that $G$ is discrete, and consider it as a category with a single object and morphisms given by the elements of $G$.

\begin{definition}
The \emph{category of} $G$-\emph{objects in} $\mathcal C$, denoted by $\mathcal C^G$, is the category of functors $G \rightarrow \mathcal C$.
\end{definition}

The key feature of a $G$-equivariant model structure, rather than just a model structure of objects equipped with a $G$-action, is that weak equivalences are defined using fixed point objects.

\begin{definition}
Given any subgroup $H$ of $G$, define the $H$-\emph{fixed points functor}, denoted by $(-)^H$, as the composite $\mathcal C^G \rightarrow \mathcal C^H \rightarrow \mathcal C$ where the first map is the restriction defined by the inclusion of $H$ into $G$, and the second map is the limit functor.
\end{definition}

\begin{definition}
Let $G$ be a discrete group.  The $G$-\emph{model structure} on $\mathcal C^G$ (if it exists) has weak equivalences the maps $X \rightarrow Y$ such that, for every subgroup $H$ of $G$, the map $X^H \rightarrow Y^H$ is a weak equivalence in $\mathcal C$, and likewise for the fibrations.
\end{definition}

\begin{remark}
Stephan works in the following more general setting.  Let $\mathcal F$ be a set of subgroups of $G$.  The $\mathcal F$-\emph{model structure} on $\mathcal C$ is defined similarly to the $G$-model structure, but where $H$ is taken to be an element of $\mathcal F$.  We are most interested in the $G$-model structure, but our results can be applied to the $\mathcal F$-model structure as well.
\end{remark}

Since $\mathcal C$ is a model category, and in particular is complete and cocomplete, we can think of it as being tensored and cotensored over the category of sets, as follows.  Let $A$ and $B$ be objects of $\mathcal C$ and $X$ a set.  Then there is a tensor
\[ X \otimes A = \coprod_X A \]
and a cotensor
\[ [X,B] = \prod_X B \]
such that there are isomorphisms
\[ \Hom_\mathcal C(X \otimes A, B) \cong \Hom_{\Sets}(X, \Hom_\mathcal C(A,B)) \cong \Hom_\mathcal C(A, [X,B]). \]

Let $A$ be an object of $\mathcal C$ and $H$ a subgroup of $G$.  Then we denote by $G/H \otimes A$ the object of $\mathcal C^G$ given composing the map $G \rightarrow \Sets$ given by sending $G$ to $G/H$ with the tensor map $- \otimes A \colon \Sets \rightarrow \mathcal C$.

\begin{lemma} \cite{stephan}
The functor $G/H \otimes - \colon \mathcal C \rightarrow \mathcal C^G$ is left adjoint to the fixed point functor $(-)^H$.
\end{lemma}

We use the following criteria for determining the existence of the $G$-model structure on $\mathcal C^G$.  Here, we use the formulation given by \cite{bmoopy} which is slightly different from the original in \cite{stephan}.

\begin{theorem} \cite{stephan}, \cite{bmoopy} \label{cellular}
Let $G$ be a group, and let $\mathcal C$ be a cofibrantly generated model category.  Suppose that, for each subgroup $H$ of $G$, the fixed point functor $(-)^H$ satisfies the following cellularity conditions:
\begin{enumerate}
\item \label{cell1} the functor $(-)^H$ preserves filtered colimits of diagrams in $\mathcal C^G$,

\item \label{cell2} the functor $(-)^H$ preserves pushouts of diagrams where one arrow is of the form
\[ G/K \otimes f \colon G/K \otimes A \rightarrow G/K \otimes B \]
for some subgroup $K$ of $G$ and $f$ a generating cofibration of $\mathcal C$, and

\item \label{cell3} for any subgroup $K$ of $G$ and object $A$ of $\mathcal C$, the induced map
\[ (G/H)^K \otimes A \rightarrow (G/H \otimes A)^K \]
is an isomorphism in $\mathcal C$.
\end{enumerate}
Then the category $\mathcal C^G$ admits the $G$-model structure.
\end{theorem}

This result was further strengthened by Bohmann, Mazur, Osorno, Ozornova, Ponto, and Yarnall via the following result.

\begin{theorem} \cite{bmoopy} \label{qe}
Suppose that $F \colon \mathcal M \rightleftarrows \mathcal N \colon R$ be a Quillen equivalence between model categories satisfying the cellularity conditions of Theorem \ref{cellular}.  Then there is an induced Quillen equivalence $F^G \colon \mathcal M^G \rightleftarrows \mathcal N^G \colon R^G$.
\end{theorem}

In fact, we have the stronger result that these model categories are actually Quillen equivalent to those described via orbit diagrams, generalizing a result of Elmendorf for topological spaces \cite{elm}.

\begin{definition}
The \emph{orbit category of} $G$ is the full subcategory $\mathcal O_G$ of the category of $G$-sets with objects the orbits $G/H$ where $H$ is a subgroup of $G$.
\end{definition}

Consider the category of functors $\mathcal O_G^{op} \rightarrow \mathcal C$, equipped with the projective model structure, where weak equivalences and fibrations are defined levelwise.  There is a functor $i \colon G \rightarrow \mathcal O_G^{op}$ sending the single object of $G$ to $G/\{e\}$ and a morphism $g$ of $G$ to the $G$-map $G/\{e\} \rightarrow G/\{e\}$ defined by $h \mapsto hg$.  The induced map $i^* \colon \mathcal G^{\orbit} \rightarrow \mathcal C^G$ has a left adjoint $i_*$ given by left Kan extension.

\begin{theorem} \cite{stephan} \label{grouporbit}
Let $G$ be a group, and let $\mathcal C$ be a cofibrantly generated model category.  Suppose that, for each subgroup $H$ of $G$, the fixed point functor $(-)^H$ satisfies the cellularity conditions of Theorem \ref{cellular}.
Then there is a Quillen equivalence
\[ i^\ast \colon \mathcal C^{\mathcal O_G^{op}} \rightleftarrows \mathcal C^G \colon i_\ast. \]
\end{theorem}

\begin{remark}
If one instead uses the $\mathcal F$-model structure, for a set $\mathcal F$ of subgroups of $G$, then it is necessary to assume that the trivial subgroup of $G$ is included in $\mathcal F$, so that the functor $i$ is defined.
\end{remark}

Combining this result with that of Theorem \ref{qe}, we get a commutative square of Quillen equivalent model categories associated to a Quillen equivalence $\mathcal C \rightleftarrows \mathcal D$ and a group $G$:
\[ \xymatrix{\mathcal C^G \ar@<.5ex>[r] \ar@<.5ex>[d] & \mathcal D^G \ar@<.5ex>[l] \ar@<.5ex>[d] \\
\mathcal C^{\mathcal O^{op}_G} \ar@<.5ex>[u] \ar@<.5ex>[r] & \mathcal D^{\mathcal O^{op}_G}. \ar@<.5ex>[u] \ar@<.5ex>[l] } \]

We conclude this section with two important families of examples.

\begin{theorem} \cite{stephan} \label{functors}
Any category of functors $\mathcal D \rightarrow \Sets$, equipped with a cofibrantly generated model structure in which the cofibrations are monomorphisms, admits the $G$-model structure.
\end{theorem}

Observe that we can apply the above theorem to any category of functors $\mathcal D \rightarrow \SSets$ simply by viewing the objects instead as functors $\mathcal D \times \Deltaop \rightarrow \Sets$.

\begin{theorem} \cite{stephan} \label{localization}
Let $\mathcal C$ be a model category which satisfies the cellularity conditions of Theorem \ref{cellular}.  Then any left Bousfield localization of $\mathcal C$ also satisfies the cellularity conditions, and hence admits the $G$-model structure.
\end{theorem}

\section{Equivariant complete Segal spaces and quasi-categories}

In this section, we consider the two models for $(\infty, 1)$-categories for which we get equivariant versions immediately from Stephan's results.  Using Theorem \ref{qe}, we can lift the Quillen equivalences between them to these equivariant versions.

We begin with the model of quasi-categories.  Recall that, given any $n \geq 1$ and $0 \leq k \leq n$, the $k$th horn of the $n$-simplex $\Delta[n]$, denoted by $V[n,k]$, is the boundary of $\Delta[n]$ with the $k$th face removed.

\begin{definition}
A \emph{quasi-category} is a simplicial set $K$ satisfying the inner horn filling conditions, i.e., that for any $n \geq 2$ and $0 <k<n$, in any diagram
\[ \xymatrix{V[n,k] \ar[r] \ar[d] & K \\
\Delta[n] \ar@{-->}[ur] &} \]
the dotted arrow lift exists.
\end{definition}

Recall the following result, proved using different methods by Joyal \cite{joyal}, Lurie \cite{lurie}, and Dugger and Spivak \cite{ds}.

\begin{theorem}
There is a model structure $\qcat$ on the category of simplicial sets in which the cofibrations are the monomorphisms and the fibrant objects are the quasi-categories.
\end{theorem}

We can apply Theorem \ref{functors} and Theorem \ref{grouporbit} to obtain the following result.

\begin{theorem}
The $G$-model structure for $\qcat$ exists, and there is a Quillen equivalence of model categories
\[ i^\ast \colon \qcat^{\mathcal O_G^{op}} \rightleftarrows \qcat^G \colon i_\ast. \]
\end{theorem}

Now we turn to the complete Segal space model.  The objects here are simplicial spaces, or functors $\Deltaop \rightarrow \SSets$.  Recall that the category of simplicial spaces can be equipped with the Reedy model structure, in which the weak equivalences are given levelwise \cite{reedy}.  In this case, it is equivalent to the injective model structure, in which the cofibrations are also defined levelwise.  Recall that the Reedy model structure is equipped with a compatible simplicial structure, so there is a mapping simplicial set $\Map(X,Y)$ between any simplicial spaces $X$ and $Y$.

\begin{definition}
A simplicial space $W$ is a \emph{Segal space} if it is Reedy fibrant and the Segal maps
\[ W_k \rightarrow \underbrace{W_1 \times_{W_0} \cdots \times_{W_0} W_1}_k \] are weak equivalences.
\end{definition}

Let $E$ denote the nerve of the groupoid with two objects and a single isomorphism between them.  Denote by $E^t$ its corresponding discrete simplicial space (the ``transpose" of the constant simplicial space at $E$).

\begin{definition}
A Segal space $W$ is \emph{complete} if the map $W_0 \rightarrow \Map(E^t, W)$ is a weak equivalence of simplicial sets.
\end{definition}

\begin{theorem}\cite[7.2]{rezk} \label{CSS}
There is a model structure $\css$ on the category of simplicial spaces, obtained as a left Bousfield localization of the Reedy model structure, which satisfies the
following properties:

\begin{enumerate}
\item the cofibrations are the monomorphisms, and

\item the fibrant objects are the complete Segal spaces.
\end{enumerate}
\end{theorem}

We obtain the following result by applying Theorems \ref{functors}, \ref{localization}, and \ref{grouporbit}.

\begin{cor}
The model category $\css$ admits the $G$-model structure $\css^G$. There is a Quillen equivalences of model categories
 \[ i^\ast \colon \css^{\mathcal O_G^{op}} \rightleftarrows \css^G \colon i_\ast \]
\end{cor}

We now turn to the comparison between the model structures $\qcat$ and $\css$.  There are two different Quillen equivalences, both due to Joyal and Tierney.

\begin{theorem} \cite[4.11]{jt} \label{qcat1}
The functor $p^* \colon \css \rightarrow \qcat$, which associates to a complete Segal space $W$ the
simplicial set $W_{*,0}$, has a left adjoint $p_!$.  This adjoint pair defines a Quillen equivalence
\[ p^* \colon \css \leftrightarrows \qcat \colon p_!. \]
\end{theorem}

The second Quillen equivalence between these two model categories
is given by a total simplicial set functor $t_! \colon \css \rightarrow \qcat$ and its right adjoint $t^!$.

\begin{theorem} \cite[4.12]{jt}
The adjoint pair
\[ t_! \colon  \css \rightleftarrows \qcat \colon t^! \]
is a Quillen equivalence.
\end{theorem}

\begin{cor}
There are two commuting squares of Quillen equivalences:
\[ \xymatrix{\qcat^G \ar@<.5ex>[r] \ar@<.5ex>[d] & \css^G \ar@<.5ex>[l] \ar@<.5ex>[d] & \qcat^G \ar@<-.5ex>[r] \ar@<.5ex>[d] & \css^G \ar@<-.5ex>[l] \ar@<.5ex>[d] \\
\qcat^{\orbit} \ar@<.5ex>[u] \ar@<.5ex>[r] & \css^{\orbit} \ar@<.5ex>[l] \ar@<.5ex>[u] & \qcat^{\orbit} \ar@<.5ex>[u] \ar@<-.5ex>[r] & \css^{\orbit}. \ar@<-.5ex>[l] \ar@<.5ex>[u]} \]
\end{cor}

\section{Equivariant simplicial categories}

In this section we consider simplicial categories.  Recall that a simplicial category is a category enriched in simplicial sets, so that between any two objects $x$ and $y$, there is a simplicial set $\Map(x,y)$, together with compatible composition.  In this case, the $G$-equivariant model structure does not follow immediately from previous results.

We start by recalling some notation.  Given any simplicial category $\mathcal C$, there is an associated \emph{category of components} $\pi_0 \mathcal C$ whose objects are the same as those of $\mathcal C$ and whose morphisms are the sets of components of the mapping spaces of $\mathcal C$.

Define the functor $U: \SSets \rightarrow \mathcal{SC}$ such that for any simplicial set $K$, the simplicial category $UK$ has two objects, $x$ and $y$, and
only nonidentity morphisms the simplicial set $K = \Hom(x,y)$.

\begin{theorem} \cite[1.1]{simpcat}
There is a cofibrantly generated model structure $\mathcal{SC}$ on the category
of small simplicial categories with the following properties.

\begin{enumerate}
\item The weak equivalences are the simplicial functors $f \colon \mathcal C
\rightarrow \mathcal D$ satisfying the following two conditions:
\begin{itemize}
\item for any objects $x$ and $y$ in $\mathcal C$, the map
\[ \Map_\mathcal C (x,y) \rightarrow \Map_\mathcal D (fx,fy) \]
is a weak equivalence of simplicial sets, and

\item the induced functor $\pi_0f \colon \pi_0 \mathcal C
\rightarrow \pi_0 \mathcal D$ on the categories of components is
an equivalence of categories.
\end{itemize}

\item A set of generating cofibrations for $\mathcal{SC}$ contains the functors
\begin{itemize}
\item $U \partial \Delta [n] \rightarrow U \Delta [n]$ for $n
\geq 0$, and

\item $\varnothing \rightarrow \{x\}$, where $\varnothing$ is the
simplicial category with no objects and $\{x\}$ denotes the
simplicial category with one object $x$ and no nonidentity
morphisms.
\end{itemize}
\end{enumerate}
\end{theorem}

To obtain a model structure for $G$-simplicial categories, we need to verify that the cellularity conditions hold.

\begin{theorem}
The model category $\mathcal{SC}$ satisfies the cellularity conditions of Theorem \ref{cellular}.
\end{theorem}

\begin{proof}
To establish condition \eqref{cell1}, we modify the argument used for categories in \cite{bmoopy}.  Let $I$ be a filtered category and $F \colon I \rightarrow \mathcal{SC}^G$ a functor.  Consider the simplicial nerve functor $N \colon \mathcal{SC} \rightarrow \SSets^{\Deltaop}$.  Thinking of a simplicial category as a special case of a simplicial object in $\mathcal Cat$, the simplicial nerve is given by levelwise nerve of categories.  Therefore, the functor $N$ commutes with filtered colimits \cite{nerve}.  Thus, we get an isomorphism
\[ N \colim_I(F(i)^H) \cong \colim_I(N(F(i))^H). \]
Since the fixed point functor $(-)^H$ is defined as a limit and the simplicial nerve is a right adjoint functor, we get an isomorphism
\[ N(F(i)^H) \cong (NF(i))^H. \]
Since finite limits and filtered colimits commute in $\Sets$ and are computed levelwise in $\SSets^{\Deltaop}$, we obtain isomorphisms
\[ \colim_I((NF(i))^H) \cong (\colim_I NF(I))^H \cong (N\colim_IF(i))^H \]
where the second isomorphism is a second application of commuting the filtered colimit with the nerve.  We again use that $N$ is a right adjoint, so that
\[ (N\colim_IF(i))^H \cong N(\colim_I F(i))^H. \]
Since the nerve functor is fully faithful, this isomorphism holds even before applying the nerve functor, which completes the proof of condition \eqref{cell1}.

Next we consider condition \eqref{cell3}.  The tensor product is given by the disjoint union of categories, and we want to show that the map
\[ \coprod_{(G/K)^H} A \rightarrow \left(\coprod_{G/K} A \right)^H \]
is an isomorphism.  The action of $H$ on $\coprod_{G/K} A$ is given by permuting the copies of $A$, since each $A$ itself has trivial $G$-action.  Therefore, this action is determined precisely by the action of $H$ on $G/K$.  The desired isomorphism follows.

Finally, we establish condition \eqref{cell2}.  Given any pushout diagram of the form
\[ \xymatrix{\coprod_{G/K} \mathcal A \ar[r] \ar[d] & \mathcal C \ar[d] \\
\coprod_{G/K} \mathcal B \ar[r] & \mathcal D} \]
with $\mathcal A \rightarrow \mathcal B$ a generating cofibration in $\mathcal{SC}$, we want to show, making use of condition \eqref{cell3}, that the diagram
\[ \xymatrix{\coprod_{(G/K)^H} \mathcal A \ar[r] \ar[d] & \mathcal C^H \ar[d] \\
\coprod_{(G/K)^H} \mathcal B \ar[r] & \mathcal D^H} \]
is again a pushout square.

We first consider the case where $\mathcal A$ is the initial simplicial category $\varnothing$ and $\mathcal B$ is a terminal simplicial category $\{x\}$.  Then in the original pushout square, the simplicial category $\mathcal D$ is obtained from the simplicial category $\mathcal C$ simply by adjoining disjoint objects indexed by the cosets $G/K$.  When we apply the fixed point functor $(-)^H$, we only adjoin those objects indexed by $(G/K)^H$, which establishes the desired pushout.

It remains to consider the case where $\mathcal A \rightarrow \mathcal B$ is of the form $U\partial \Delta[n] \rightarrow U \Delta [n]$ for any $n \geq 0$.  In this case, $\mathcal D$ is obtained from $\mathcal C$ by attaching, for each coset $G/K$, an $n$-simplex of morphisms to the appropriate mapping space of $\mathcal C$, then freely adjoining all necessary composites with other mapping spaces.  Since this gluing is done in an equivariant manner, the adjoined simplices have $H$-action as specified by the $H$-action on $(G/K)^H$.  Since the composites are also included in a manner compatible with the $G$-action, again the only fixed points of $\mathcal D$ by the action of $H$ can be those of $\mathcal C$ or those given by simplices indexed by $(G/K)^H$, as desired.
\end{proof}

\begin{cor}
The model structure $\mathcal{SC}^G$ exists and there is a Quillen equivalence
\[ p^* \colon \mathcal{SC}^{\orbit} \rightleftarrows \mathcal{SC}^G \colon i_*. \]
\end{cor}

We can now relate these model structures to the other models for equivariant $(\infty, 1)$-categories from the previous section.  We use a direct Quillen equivalence between $\mathcal{SC}$ and $\qcat$, defined using the coherent nerve functor $\widetilde{N} \colon \mathcal{SC} \rightarrow \qcat$, originally due to Cordier and Porter \cite{cp}. Given a simplicial category $X$ and the simplicial resolution $C_*[n]$ of the category $[n]=(0 \rightarrow \cdots \rightarrow n)$, the coherent nerve $\widetilde{N}(X)$ is defined by
\[ \widetilde{N}(X)_n = \Hom_{\mathcal SC}(C_*[n],X). \]
This functor has a left adjoint $J \colon \qcat \rightarrow \mathcal{SC}$.

\begin{theorem} \cite{ds}, \cite{joyal3}, \cite{lurie}
The adjoint pair
\[ \xymatrix@1{J \colon \qcat \ar@<.5ex>[r] & \mathcal{SC} \colon \widetilde{N} \ar@<.5ex>[l]}
\] is a Quillen equivalence.
\end{theorem}

\begin{cor}
There is a commuting square of Quillen equivalences
\[ \xymatrix{\qcat^G \ar@<.5ex>[r] \ar@<.5ex>[d] & \mathcal{SC}^G \ar@<.5ex>[l] \ar@<.5ex>[d] \\
\qcat^{\orbit} \ar@<.5ex>[r] \ar@<.5ex>[u] & \mathcal{SC}^{\orbit}. \ar@<.5ex>[l] \ar@<.5ex>[u]} \]
\end{cor}

\section{Equivariant Segal categories}

Lastly, we turn to the model of Segal categories.  In this case, we consider two different model structures with the same weak equivalences.

\begin{definition}
A simplicial space $X$ is a \emph{Segal precategory} if $X_0$ is a discrete simplicial set.  It is a \emph{Segal category} if additionally the Segal maps
\[ X_k \rightarrow \underbrace{X_1 \times_{X_0} \cdots \times_{X_0} X_1}_k \] are weak equivalences for all $k \geq 2$.
\end{definition}

For any Segal category $X$, we can take its \emph{objects} to be the set $X_0$.  Its \emph{mapping spaces} $\map_X(x,y)$ are given by the fibers of the map $(d_1, d_0) \colon X_1 \rightarrow X_0 \times X_0$ over a given pair of objects $(x,y)$.  With notions of weak composition, one can define \emph{homotopy equivalences} as in \cite{rezk} and hence a \emph{homotopy category} $\Ho(X)$ associated to $X$.

Given a Segal precategory $X$, there is a functorial construction of a Segal category $LX$ which is weakly equivalent to $X$ in the Segal space model structure \cite[\S 5]{thesis}.

The inclusion functor from the category of Segal precategories into the category of simplicial spaces has a left adjoint which we denote by $(-)_r$. Recall that, for a simplicial set $K$ we denote again by $K$ the constant simplicial space, and by $K^t$ the discrete simplicial space defined by the set of $n$-simplices of $K$ in degree $n$.

\begin{theorem} \cite[3.2]{fibrant}, \cite[5.1, 5.13]{thesis}, \cite{pel}
There exists a cofibrantly generated model structure $\Secat_c$ on the category of Segal precategories satisfying the following conditions.

\begin{itemize}
\item The weak equivalences are the \emph{Dwyer-Kan equivalences}, or maps $X \rightarrow Y$ such that
\begin{itemize}
\item for any objects $x, y \in X_0$, $\map_{LX}(x,y) \rightarrow \map_{LY}(fx,fy)$ is a weak equivalence of simplicial sets, and

\item the induced functor $\Ho(LX) \rightarrow \Ho(LY)$ is an equivalence of categories.
\end{itemize}

\item The fibrant objects are the Reedy fibrant Segal categories.

\item Cofibrations are the monomorphisms.

\item A set of generating cofibrations for $\Secat_c$ is given by
\[ I_c= \{(\partial \Delta [m] \times \Delta [n]^t \cup \Delta [m] \times \partial \Delta [n]^t)_r \rightarrow (\Delta [m] \times \Delta [n]^t)_r \} \]
for all $m \geq 0$ when $n \geq 1$ and for $n=m=0$.
\end{itemize}
\end{theorem}

Observe that this model structure actually satisfies the conditions of Theorem \ref{functors}, since its objects are presheaves and the cofibrations are monomorphisms.  However, since we have a restriction on those presheaves, namely that the degree zero space be discrete, some features of the model structure are less intuitive.  Therefore, we include a complete proof.

\begin{theorem}
The model category $\Secat_c$ satisfies the cellularity conditions of Theorem \ref{cellular}.
\end{theorem}

\begin{proof}
To show that condition \eqref{cell1} is satisfied, we need only observe that finite limits and filtered colimits commute in the category of sets.  Since limits are computed levelwise in the category of Segal precategories, and the fixed point functor $(-)^H$ is defined as a finite limit, the desired condition holds.

Establishing condition \eqref{cell3} is similar to the case of simplicial categories.

It remains to show that condition \eqref{cell2} holds.   Let $A \rightarrow B$ be a generating cofibration for the Reedy model structure on simplicial spaces, of the form
\[ \partial \Delta [m] \times \Delta[n]^t \cup \Delta[m] \times \partial \Delta[n]^t \rightarrow \Delta[m] \times \Delta[n]^t \]
for some $m, n \geq 0$.  We know that the Reedy model structure satisfies the desired condition for these generating cofibrations, using Theorem \ref{functors}.

Any generating cofibration of $\Secat_c$ is of the form $A_r \rightarrow B_r$, where $A \rightarrow B$ is as above and $(-)_r$ denotes the reduction functor.   Suppose that we have a pushout diagram
\[ \xymatrix{\coprod_{G/K} A_r \ar[r] \ar[d] & X \ar[d] \\
\coprod_{G/K} B_r \ar[r] & Y.} \]
Consider also the pushout
\[ \xymatrix{\coprod_{G/K} A \ar[r] \ar[d] & X \ar[d] \\
\coprod_{G/K} B \ar[r] & Y'} \]
taken in the category of simplicial spaces.
If we assume that $X$ is a Segal precategory, so that $X=X_r$, then the fact that the reduction functor is a left adjoint implies that $(Y')_r \simeq Y$.  Therefore, in the diagram
\[ \xymatrix{\coprod_{(G/K)^H} A \ar[r] \ar[d] & \coprod_{(G/K)^H} A_r \ar[r] \ar[d] & X^H \ar[d] \\
\coprod_{(G/K)^H} B \ar[r] & \coprod_{(G/K)^H} B_r \ar[r] & Y^H} \]
the left square and the large rectangle are both pushouts; therefore, the right square is also a pushout.
\end{proof}

\begin{cor}
The model structure $\Secat_c^G$ exists and there is a Quillen equivalence
\[ i^* \colon \Secat_c^{\orbit} \rightleftarrows \Secat_c^G \colon i_*. \]
\end{cor}

For the purposes of comparison with other models, we need another model structure with the same weak equivalences but different fibrations and cofibrations.  To define a generating set of cofibrations, we require the following construction.

For $m \geq 1$ and $n \geq 0$, define $P_{m,n}$ to be the pushout of the diagram
\[ \xymatrix{\partial \Delta [m] \times \Delta[n]^t_0 \ar[r] \ar[d] & \partial \Delta [m] \times \Delta [n]^t \ar[d] \\
\Delta [n]^t_0 \ar[r] & P_{m,n}.} \]
If $m=0$, then we define $P_{m,0}$ to be the empty simplicial space.  For all $m \geq 0$
and $n \geq 1$, define $Q_{m,n}$ to be the pushout of the diagram
\[ \xymatrix{\Delta [m] \times \Delta [n]^t_0 \ar[r] \ar[d] & \Delta [m] \times \Delta [n]^t \ar[d] \\
\Delta [n]^t_0 \ar[r] & Q_{m,n}.} \]
For each $m$ and $n$, the map $\partial \Delta [m] \times \Delta [n]^t \rightarrow \Delta[m] \times \Delta[n]^t$ induces a map
$i_{m,n}:P_{m,n} \rightarrow Q_{m,n}$.  Note
that when $m \geq 2$ this construction gives exactly the same
objects as those given by reduction, namely that $P_{m,n}$ is
precisely $(\partial \Delta [m] \times \Delta [n]^t)_r$ and likewise
$Q_{m,n}$ is precisely $(\Delta [m] \times \Delta [n]^t)_r$.

\begin{theorem} \cite[7.1]{thesis}
There is a model structure $\Secat_f$ on the category of Segal precategories with the following properties.

\begin{itemize}
\item The weak equivalences are the Dwyer-Kan equivalences.

\item The cofibrations are the maps which can be formed by taking
iterated pushouts along the maps of the set
\[ I_f = \{i_{m,n}:P_{m,n} \rightarrow Q_{m,n} \mid m,n \geq 0\}. \]

\item The fibrant objects are the Segal categories which are fibrant in the projective model structure on simplicial spaces.
\end{itemize}
\end{theorem}

\begin{theorem}
The model category $\Secat_f$ satisfies the cellularity conditions of Theorem \ref{cellular}.
\end{theorem}

\begin{proof}
Since the objects and weak equivalences in $\Secat_c$ and $\Secat_f$ are the same, the cellularity conditions \eqref{cell1} and \eqref{cell3} continue to hold for $\Secat_f$.

We need only prove that condition \eqref{cell2} holds.  Consider a generating cofibration $A \rightarrow B$ for the projective model structure on simplicial spaces, which has the form
\[ \partial \Delta[m] \times \Delta[n]^t \rightarrow \Delta[m] \times \Delta[n]^t \]
for some $m, n \geq 0$.  This model structure satisfies the cellularity conditions by Theorem \ref{functors}.

For any $m \geq 2$, the map $P_{m,n} \rightarrow Q_{m,n}$ coincides with $A_r \rightarrow B_r$, so the argument given for $\Secat_c$ applies.  When $m=0$, the map is an isomorphism.  Therefore, we need only consider the case when $m=1$.

The pushout diagram defining $P_{1,n}$ can be rewritten as
\[ \xymatrix{\Delta[n]^t_0 \amalg \Delta[n]^t_0 \ar[r] \ar[d] & \Delta[n]^t \amalg \Delta[n]^t \ar[d] \\
\Delta[n]^t_0 \ar[r] & P_{1,n}.} \]
The top horizontal map is the inclusion, and the left horizontal map is the fold map, so it follows that $P_{1,m} = \Delta[n]^t$, which coincides with $(\Delta[n]^t)_r$ since $\Delta[n]^t_0$ is already discrete.  Since $Q_{1,n} = (\Delta[1] \times \Delta[n]^t)_r$, we simply need to consider the map $(\Delta[n]^t)_r \rightarrow (\Delta[1] \times \Delta[n]^t)_r$.  But this map is the reduction of a cofibration in the projective model structure, so we can apply the same argument as the one used for $\Secat_c$ to the diagram
\[ \xymatrix{\coprod_{(G/K)^H} \Delta[n]^t \ar[r] \ar[d] & \coprod_{(G/K)^H} (\Delta[n]^t)_r \ar[r] \ar[d] & X^H \ar[d] \\
\coprod_{(G/K)^H} \Delta[1] \times \Delta[n]^t \ar[r] & \coprod_{(G/K)^H} (\Delta[1] \times \Delta[n]^t)_r \ar[r] & Y^H.} \]
It follows that the right-hand square is a pushout, which is what we needed to prove.
\end{proof}

\begin{cor}
The model structure $\Secat_f^g$ exists and there is a Quillen equivalence
\[ i^* \colon \Secat_f^{\orbit} \rightleftarrows \Secat_f^G \colon i_*. \]
\end{cor}

We now have a number of comparisons with other models, beginning with the one connecting $\Secat_c$ and $\Secat_f$.

\begin{prop} \cite[7.5]{thesis}
The identity functor induces a Quillen equivalence
\[ I \colon \Secat_f \rightleftarrows \Secat_c \colon J. \]
\end{prop}

\begin{cor}
The identity functor induces a commutative square of Quillen equivalences
\[ \xymatrix{\Secat_f^G \ar@<.5ex>[r] \ar@<.5ex>[d] & \Secat_c^G \ar@<.5ex>[l] \ar@<.5ex>[d] \\
\Secat_f^{\orbit} \ar@<.5ex>[r] \ar@<.5ex>[u] & \Secat_c^{\orbit}. \ar@<.5ex>[l] \ar@<.5ex>[u]} \]
\end{cor}

Next, we consider the comparison between simplicial categories and Segal categories.  The functor $N$ taking a simplicial category to its simplicial nerve, a Segal precategory (in fact, a Segal category) has a left adjoint which we denote by $F$.

\begin{theorem} \cite[8.6]{thesis}
The Quillen pair
\[ F \colon \Secat_f \rightleftarrows \mathcal{SC} \colon N  \] is
a Quillen equivalence.
\end{theorem}

\begin{cor}
There is a commutative square of Quillen equivalences
\[ \xymatrix{\Secat_f^G \ar@<.5ex>[r] \ar@<.5ex>[d] & \mathcal{SC}^G \ar@<.5ex>[l] \ar@<.5ex>[d] \\
\Secat_f^{\orbit} \ar@<.5ex>[r] \ar@<.5ex>[u] & \mathcal{SC}^{\orbit}. \ar@<.5ex>[l] \ar@<.5ex>[u]} \]
\end{cor}

Next, we consider the inclusion functor $I$ from the category of Segal precategories to the category of simplicial spaces, which has a right adjoint $R$.

\begin{theorem} \cite[6.3]{thesis}
The adjoint pair
\[ I \colon \Secat_c \rightleftarrows \css \colon R  \]
is a Quillen equivalence.
\end{theorem}

\begin{cor}
There is a commutative square of Quillen equivalences
\[ \xymatrix{\Secat_c^G \ar@<.5ex>[r] \ar@<.5ex>[d] & \css^G \ar@<.5ex>[l] \ar@<.5ex>[d] \\
\Secat_c^{\orbit} \ar@<.5ex>[r] \ar@<.5ex>[u] & \css^{\orbit}. \ar@<.5ex>[l] \ar@<.5ex>[u]} \]
\end{cor}

Similarly to the comparison between complete Segal spaces and quasi-categories, Joyal and Tierney prove that
there are also two different Quillen equivalences directly between $\qcat$ and $\Secat_c$.  The first of these functors is analogous
to the pair given in Theorem \ref{qcat1}; the functor $j^* \colon \Secat_c \rightarrow \qcat$ assigns to a Segal
precategory $X$ the simplicial set $X_{*0}$.  Its left adjoint is
denoted $j_!$.

\begin{theorem} \cite[5.6]{jt}
The adjoint pair
\[ j_! \colon \qcat \rightleftarrows \Secat_c \colon j^* \]
is a Quillen equivalence.
\end{theorem}

The second Quillen equivalence between these two model categories
is given by the map $d^* \colon \Secat_c \rightarrow \qcat$, which
assigns to a Segal precategory its diagonal, and its right adjoint
$d_*$.

\begin{theorem} \cite[5.7]{jt}
The adjoint pair
\[ d^* \colon \Secat_c \rightleftarrows \qcat \colon d_* \]
is a Quillen equivalence.
\end{theorem}

\begin{cor}
There are two commutative squares of Quillen equivalences
\[ \xymatrix{\qcat^G \ar@<.5ex>[r] \ar@<.5ex>[d] & \Secat_c^G \ar@<.5ex>[l] \ar@<.5ex>[d] & \qcat^G \ar@<-.5ex>[r] \ar@<.5ex>[d] & \Secat_c^G \ar@<-.5ex>[l] \ar@<.5ex>[d] \\
\qcat^{\orbit} \ar@<.5ex>[u] \ar@<.5ex>[r] & \Secat_c^{\orbit} \ar@<.5ex>[l] \ar@<.5ex>[u] & \qcat^{\orbit} \ar@<.5ex>[u] \ar@<-.5ex>[r] & \Secat_c^{\orbit}. \ar@<-.5ex>[l] \ar@<.5ex>[u]} \]
\end{cor}

\section{Actions of simplicial groups}

In this section, we consider the more general case of actions by any simplicial group $G$.  Given a model category $\mathcal C$, to make sense of the category $\mathcal C^G$, we need to require that $\mathcal C$ have the structure of a simplicial model category and consider simplicial functors $G \rightarrow \mathcal C$ where $G$ is regarded as a simplicial category with one object.

Stephan proves the following result for the topological case, but his proof holds in the simplicial case as well.  He restricts to the case of compact Lie groups, partially because they are the case of most interest in equivariant homotopy theory, but also because they satisfy important cofibrancy conditions.  In particular, if $G$ is a compact Lie group, then the spaces $G/H$ and $(G/H)^K$ are CW complexes for any closed subgroups $H$ and $K$ of $G$.  In our case, the analogous statement is that the simplicial sets $G/H$ and $(G/H)^K$ are cofibrant, which holds automatically since all simplicial sets are cofibrant.

\begin{prop} \label{simplicial}
Suppose that $G$ is a simplicial group and $\mathcal C$ is a cofibrantly generated simplicial model category.  Then we have the following.
\begin{enumerate}
\item The category of $G$-objects in $\mathcal C$ admits the $G$-model structure if the conditions of Theorem \ref{cellular} hold.  In this case, $\mathcal C^G$ is a simplicial model category.

\item The orbit category model category $\mathcal C^{\mathcal O^{op}_G}$ exists and is a simplicial model category.

\item There is a Quillen equivalence $\mathcal C^{\mathcal O^{op}_G} \rightleftarrows \mathcal C^G$.
\end{enumerate}
\end{prop}

Since $\css$ is known to be a simplicial model category \cite[7.2]{rezk}, we have the following.

\begin{cor}
The simplicial model categories $\css^G$ and $\css^{\mathcal O^{op}_G}$ exist and there is a Quillen equivalence
\[ i^* \colon \css^{\orbit} \rightleftarrows \css^G \colon i_*. \]
\end{cor}

\begin{proof}
We have already proved that $\css$ satisfies the cellularity conditions.  Therefore, the result follows immediately from Proposition \ref{simplicial}.
\end{proof}

However, we can also consider $\Secat_c$; we include here a proof that it is a simplicial model category.

\begin{prop}
The model category $\Secat_c$ has the structure of a simplicial model category.
\end{prop}

\begin{proof}
We need to show that the axioms (SM6) and (SM7) for a simplicial model category hold \cite[9.1.6]{hirsch}.  We begin with (SM6).  Suppose that $X$ and $Y$ are Segal precategories and $K$ is a simplicial set.  Recall that the left adjoint to the inclusion functor from Segal precategories to simplicial spaces is the reduction functor $(-)_r$.  Define $X \otimes K = (X \times K)_r$, where the product is taken in simplicial spaces.  Then define $\Map(X,Y)$ and $Y^K$ just as for simplicial spaces; for the latter, observe that if $Y_0$ is discrete, then so is $(Y^K)_0 = (Y_0)^K$.  Then one can verify the necessary isomorphisms to verify (SM6).

To check axiom (SM7), we need to show that if $i \colon A \rightarrow B$ is a cofibration and $p \colon X \rightarrow Y$ is a fibration in $\Secat_c$, then the pullback-corner map
\[ \Map(B,X) \rightarrow \Map(A,X) \times_{\Map(A,Y)} \Map(B,Y) \]
is a fibration of simplicial sets which is a weak equivalence if either $i$ or $p$ is.  Since we have defined mapping spaces to be the same as in the category of simplicial spaces, we need only verify that a cofibration or fibration in $\Secat_c$ is still a cofibration or fibration, respectively, in the Reedy model structure on simplicial spaces.  Since cofibrations are exactly the monomorphisms in both categories, the case of cofibrations is immediate.

Suppose, then, that $p$ is a fibration in $\Secat_c$, so that is has the right lifting property with respect to monomorphisms between Segal precategories which are also Dwyer-Kan equivalences.  In particular, it has the right lifting property with respect to monomorphisms which are levelwise weak equivalences of simplicial sets.  Suppose that $A \rightarrow B$ is an acyclic cofibration in the Reedy model structure.  Then $\pi_0(A_0) \cong \pi_0(B_0)$, so $(A_r)_0 \cong (B_r)_0$.  Therefore the map $A_r \rightarrow B_r$ is still a levelwise weak equivalence and monomorphism, and in particular a weak equivalence in $\Secat_c$.  Therefore we obtain a lifting
\[ \xymatrix{A \ar[r] \ar[d]_\simeq & A_r \ar[r] \ar[d] & X \ar[d]^p \\
B \ar[r] & B_r \ar@{-->}[ur] \ar[r] & Y.} \]
Therefore, $p$ is also a fibration in the Reedy model structure on simplicial spaces.  It follows that $\Secat_c$ satisfies axiom (SM7) and is a simplicial model category.
\end{proof}

Since we have already verified that $\Secat_c$ satisfies the cellularity conditions, we have the following result.

\begin{cor}
Let $G$ be a simplicial group.  The simplicial model categories $\Secat_c^G$ and $\Secat_c^{\mathcal O^{op}_G}$ exist and there is a Quillen equivalence
\[ i^* \colon \Secat_c^{\orbit} \rightleftarrows \Secat_c^G \colon i_*. \]
\end{cor}

Using another application of Theorem \ref{qe}, or rather, an analogue for simplicial model categories, we obtain the following.

\begin{cor}
Let $G$ be a simplicial group.  Then there are Quillen equivalences of simplicial model categories
\[ \xymatrix{\Secat_c^G \ar@<.5ex>[r] \ar@<.5ex>[d] & \css^G \ar@<.5ex>[l] \ar@<.5ex>[d] \\
\Secat_c^{\orbit} \ar@<.5ex>[r] \ar@<.5ex>[u] & \css^{\orbit}. \ar@<.5ex>[l] \ar@<.5ex>[u]} \]
\end{cor}


\begin{thebibliography}{99}

\bibitem{fibrant}
Julia E.\ Bergner, A characterization of fibrant Segal categories, \emph{Proc.\ Amer.\ Math.\ Soc.\ } 135 (2007) 4031--4037.

\bibitem{simpcat}
Julia E.\ Bergner, A model category structure on the category of simplicial categories, \emph{Trans.\ Amer.\ Math.\ Soc.} 359 (2007), 2043--2058.

\bibitem{thesis}
Julia E.\ Bergner, Three models for the homotopy theory of homotopy theories, \emph{Topology} 46 (2007), 397--436.

%

\bibitem{bmoopy}
Anna Marie Bohmann, Kristen Mazur, Ang\'elica Osorno, Viktoriya Ozornova, Kate Ponto, and Carolyn Yarnall, A model structure on $G\mathcal Cat$, preprint available at math.AT/1311.4605.

\bibitem{cp}
J.M.\ Cordier and T.\ Porter, Vogt's theorem on categories of homotopy coherent diagrams, \emph{Math.\ Proc.\ Camb.\ Phil.\ Soc.\ } (1986), 100, 65--90.

\bibitem{dm}
Emanuele Dotto and Kristian Moi, Homotopy theory of $G$-diagrams and equivariant excision, preprint available at math.AT/1403.6101.

\bibitem{ds}
Daniel Dugger and David I.\ Spivak, Rigidification of quasicategories, \emph{Algebr.\ Geom.\ Topol.\ } 11 (2011) 225--261.

\bibitem{dk}
T.\ Dyckerhoff and M.\ Kapranov, Higher Segal spaces I, preprint available at math.AT/1212.3563.

\bibitem{elm}
A.D.\ Elmendorf, Systems of fixed point sets.  \emph{Trans.\ Amer.\ Math.\ Soc.} 277 (1983), no.\ 1, 275-–284.

\bibitem{hirsch}
Philip S.\ Hirschhorn, \emph{Model Categories and Their Localizations, Mathematical Surveys and Monographs 99}, AMS, 2003.

\bibitem{joyal3}
A. Joyal, Simplicial categories vs quasi-categories, in
preparation.

\bibitem{joyal}
A.\ Joyal, The theory of quasi-categories I, in preparation.

\bibitem{jt}
Andr\'{e} Joyal and Myles Tierney, Quasi-categories vs Segal spaces, \emph{Contemp.\ Math.\ } 431 (2007) 277--326.

\bibitem{nerve}
Stephen Lack, Re: classifying functor and colimits, Category theory list, http://permalink.gmane.org/gmane.science.mathematic
s.categories/3407.

\bibitem{lurie}
Jacob Lurie, \emph{Higher topos theory. Annals of Mathematics Studies}, 170. Princeton University Press, Princeton, NJ, 2009.


\bibitem{pel}
Regis Pelissier, Cat\'egories enrichies faibles, preprint available at math.AT/0308246.

\bibitem{reedy}
C.L.\ Reedy, Homotopy theory of model categories, unpublished manuscript, available at http://www-math.mit.edu/\verb1~1psh.


\bibitem{rezk}
Charles Rezk, A model for the homotopy theory of homotopy theory, \emph{Trans.\ Amer.\ Math.\ Soc.\ }, 353(3), 973--1007.

\bibitem{stephan}
Marc Stephan, On equivariant homotopy theory for model categories, preprint available at math.AT/1308.0856.

\end{thebibliography}
\end{document}